\documentclass[11pt,leqno]{amsproc}
\usepackage{amsmath,amsthm,amssymb,amsfonts}
\usepackage{graphicx}
\usepackage{color}
\usepackage[active]{srcltx}

\newtheorem{theorem}{Theorem}[section]
\newtheorem{lemma}[theorem]{Lemma}

\newtheorem{proposition}[theorem]{Proposition}
\newtheorem{corollary}[theorem]{Corollary}
\theoremstyle{definition}
\newtheorem{definition}[theorem]{Definition}

\numberwithin{equation}{section}

\def\N{{\mathbb N}}

\def\R{{\mathbb R}}

\def\C{{\mathbb C}}

\def\seg{{\ \ \ \  \ \  }}
\def\sem{{\ \ \ \ \ \  }}

\DeclareMathOperator{\rea}{Re}

\title[The Bishop-Phelps-Bollob\'{a}s property]{The Bishop-Phelps-Bollob{\'a}s property  for numerical \\ radius   of  operators on  $L_1 (\mu)$}

\author[M.D. Acosta]{Mar\'{\i}a D. Acosta}
\address{Universidad de Granada, Facultad de Ciencias,
Departamento de An\'{a}lisis Matem\'{a}tico, 18071 Granada, Spain}
\email{dacosta@ugr.es}

\author[M. Fakhar]{Majid  Fakhar}
\address{Department of Mathematics, University of Isfahan, Isfahan, Iran; and
	School of Mathematics, Institute for Research in Fundamental Sciences (IPM),
	P.O. Box 19395-5746, Tehran, Iran}
\email{Email: fakhar@sci.ui.ac.ir}

\author[M. Soleimani]{Maryam Soleimani-Mourchehkhorti}
\address{Department of Mathematics, University of Isfahan, Isfahan, Iran, 81745-163}
\email{m.soleymanei@sci.ui.ac.ir}

\thanks{	The  first  author was  supported  by MTM2012-31755, MTM2015-65020-P,  Junta de Andaluc\'{\i}a
	P09-FQM--4911 and FQM--185.  The second author was supported by a grant from IPM (No. 94550414).}



\begin{document}

\dedicatory{\it  Dedicated to the memory of Joe Diestel}

\begin{abstract}
	In this paper, we introduce the notion of the Bishop-Phelps-Bollob{\'a}s
	property for numerical radius (BPBp-$\nu$) for a subclass of  the space of  bounded linear operators. Then, we show that certain subspaces of $\mathcal{L}(L_1(\mu))$ have the  BPBp-$\nu$ for every finite measure $\mu $. As a consequence we deduce that
	the subspaces of finite-rank operators, compact operators and weakly compact operators on
	$L_1(\mu)$ have the BPBp-$\nu$.
\end{abstract}

\maketitle

\baselineskip=.65cm

\section{Introduction}
In this paper, we provide  a version of Bishop-Phelps-Bollob{\'a}s theorem for numerical radius  for operators. To recall such result  we introduce some notation. For a Banach space $X$,  $B_X$ and
$S_X$ will be the closed unit ball and the unit sphere of $X$, respectively.  We will denote by
$X^*$ the topological dual of $X$ and by  $\mathcal{L}(X)$ the space of  bounded linear
operators on $X$  endowed with the operator norm.  The symbols $\mathcal{F}(X)$,
$\mathcal{K}(X)$ and $\mathcal{WC}(X)$ denote the spaces of finite-rank operators, compact
operators and weakly compact operators on $X$, respectively. It is well known that
$\mathcal{F}(X) \subset \mathcal{K}(X) \subset \mathcal{WC}(X)$. Throughout this paper the normed
spaces will be either real or complex.

Bishop-Phelps-Bollob{\'a}s theorem states that for any Banach space $X$, given $0 < \varepsilon < 1$,
and $(x,x^*) \in B_X \times  S_{X^*}$ such that $\vert  x^*(x) -1\vert <
\frac{\varepsilon^2}{2}$, there is a pair  $(y,y^*) \in S_X \times  S_{X^*}$ satisfying
$$
\Vert y-x \Vert < \varepsilon, \  \Vert y^*-x^* \Vert < \varepsilon \ \  \  \text{and}  \ \  \
y^*(y)=1
$$
(see for instance \cite{Bol},  \cite[Theorem 16.1]{Bo-Du} or \cite[Corollary 2.4]{CKMMR}).

After some interesting papers about denseness of the set of norm attaining operators, in 2008 it
was initiated the study of versions of Bishop-Phelps-Bollob{\'a}s Theorem for operators \cite{AAGM}.
More recently it was considered the problem of obtaining versions of such results for numerical
radius of operators (see \cite[Definition 1.2]{GuiKo}). We just mention that the numerical radius of an
operator is a continuous semi-norm in the space $\mathcal{L}(X)$ for every Banach space $X$.

Guirao and Kozhushkina proved that the spaces $c_0$ and $\ell_1$ satisfy the Bishop-Phelps-Bollob\'{a}s property for numerical radius (BPBp-$\nu$) in the real case as well as in the
complex case \cite{GuiKo}. Falc\'{o} showed the same result for $L_1 (\R)$  in the real case
\cite[Theorem 9]{Fal}.   Choi, Kim, Lee and Mart\'{i}n extended the previous result to $L_1
(\mu)$ for any positive measure $\mu$ \cite[Theorem 9]{KLM}. Avil\'{e}s, Guirao and Rodr\'{i}guez
provided sufficient conditions on a compact Hausdorff space $K$ in order that $C(K)$ has the
BPBp-$\nu$ in the real case \cite[Theorem 2.2]{AGR}. For
instance, a metrizable space $K$ satisfies  the previous condition \cite[Theorem 3.2]{AGR}. It is
an open problem whether or not such result is satisfied for any compact Hausdorff space $K$ in
the real case. In the complex case there are no results until now for $C(K)$ spaces.

In this paper,  motivated by Definition 1.2 of \cite{GuiKo}, we introduce   the notion of  the
BPBp-$\nu$ for  subspaces of the space of bounded linear operators.
A Banach space $X$  satisfies the BPBp-$\nu$,     introduced in \cite{GuiKo},  if and only if the space  $\mathcal{M}= \mathcal{L}(X)$ satisfies the BPBp-$\nu$ (Definition \ref{d-BPBP-num-radius}). Then,   we give some sufficient conditions on   a  subspace  $ \mathcal{M} $ of $ \mathcal{L}( L_1 (\mu) )$ to satisfy  the  BPBp-$\nu$, for any finite measure $\mu$.
More precisely,  we show that  $ \mathcal{M} $ has the BPBp-$\nu$ if $ \mathcal{M} $    contains   the space  of finite-rank operators on $ L_1 (\mu)$,   is  contained in the  class of representable operators   on $L_1 (\mu)$ (see  Definition \ref{d-R})  and  $T_{\vert A} \in  \mathcal{M} $
for every $T \in  \mathcal{M} $  and any measurable set $A$, where  $T_{\vert A}$ is the operator on $L_1 (\mu)$ given by $T _{\vert A} (f)= T( f
\chi_A)$ for all $f \in L_1 (\mu)$.  As a consequence of the main result we obtain that  for any $\sigma$-finite measure $\mu$,  the spaces of finite-rank  operators,  compact operators and weakly
compact operators on $L_1 (\mu)$  have   the BPBp-$\nu$.  The results are valid in the
real as well as in the complex case.

\section{Bishop-Phelps-Bollob\'{a}s theorem for numerical radius for some  classes of  operators on
	$L_1(\mu)$}

If $X$ is a Banach space and $T \in \mathcal{L}(X)$, we recall that the {\it numerical radius }
of $T$, $\nu(T)$, is defined by
$$
\nu (T)= \sup \bigl\{ \vert x^* (T(x)) \vert : x\in S_X, x^* \in S_{X^*}, x^* (x)=1\bigr\}.
$$
In general the numerical radius is a semi-norm on $\mathcal{L}(X)$
satisfying $\nu (T) \le \Vert T \Vert$ for each  $T \in \mathcal{L}(X)$. The numerical index of $X$,
$n(X)$ is defined by
$$
n(X)= \inf \{ \nu(T) :  T \in  S_{\mathcal{L} (X)}\}.
$$
Hence, $n(X)$ is the greatest constant $t$ such that  $t\Vert T \Vert \le \nu (T)$ for each $T \in
\mathcal{L}(X)$. It is  always satisfied that $0 \le n(X) \le 1$ and, in case that $n(X)=1$, it is said that
$X$ has {\it numerical index equal to $1$}.  In such case it is satisfied that $\nu (T)= \Vert T \Vert $ for
each $T \in \mathcal{L}(X)$.
It is well known that the spaces $L_1 (\mu)$ and $C(K)$ have numerical index
equal to $1$ for any measure $\mu$ and any compact Hausdorff space $K$ \cite[Theorem 2.2]{DMPW}.

Guirao and Kozhushkina  \cite{GuiKo} introduced the definition of the BPBp-$\nu$. We will use a  little  different concept  by admitting subclasses of the space of bounded linear operators on a Banach space $X$.

\begin{definition}
	\label{d-BPBP-num-radius}
	Let $X$ be a Banach space and $\mathcal{M}$  a  subspace of $ \mathcal{L}(X)$. We will say that
	$\mathcal{M}$ has the {\it Bishop-Phelps-Bollob\'{a}s property for numerical radius} (BPBp-$\nu$) if for every	$0 <\varepsilon< 1$, there is $\eta(\varepsilon)>0$ such that whenever $S\in \mathcal{M}$, $\nu (S)=1$,
	$x_0\in S_X$ and  $x_0^*\in S_{X^*}$  are such that  $x_0^*(x_0)=1$ and   $ \vert x_0^*(S(x_0))
	\vert
	> 1  - \eta(\varepsilon)$, there are $T\in \mathcal{M} $,  $x_1 \in S_X$ and $x_1^* \in S_{X^*}$
	such that
	\begin{enumerate}
		\item[i)]   $x_1^*(x_1) = 1$,
		\item[ii)] $ \vert  x_1^*(T(x_1)) \vert =  \nu  (T)=1$,
		\item[iii)]  $  \nu  (T-S)    < \varepsilon$, $ \Vert  x_1-x_0 \Vert <\varepsilon   $ and
		$\Vert  x_1^*-x_0^* \Vert <\varepsilon $.
	\end{enumerate}
\end{definition}

Let us notice  that for spaces with numerical index equal to one, Definition \ref{d-BPBP-num-radius} can be reformulated by using the usual norm of the space $\mathcal{L}(X)$  instead of the numerical radius.

The following simple technical lemmas will be useful.  Next lemma is a straightforward consequence of \cite[Lemma 3.3]{AAGM}.

\begin{lemma}
	\label{lema-compl-z-beta}
	Assume that $\{ z_k: k \in \N \} \subset \{ z \in \C: \vert z \vert \le 1\}$ and $\{\beta_k: k \in \N \}
	\subset \C$  satisfies that  $ \sum_{ k=1}^\infty  \vert \beta_k \vert  = 1 $. If $ 0 < \varepsilon < 1 $ and
	$ \rea  ( \sum_{ k=1}^\infty \beta_k z_k ) > 1 - {\varepsilon}^2  $,
	then
	$$
	\sum_{k \in B} |\beta_k|  > 1 - \varepsilon,
	$$
	where $ B = \{ k \in \N :  \rea (\beta_k z_k) > ( 1 - \varepsilon  ) | \beta_k| \} $.
\end{lemma}

Next result is a generalization of Lemma \ref{lema-compl-z-beta} to
$L_1 (\mu)$. Also it  extends \cite[Lemma 2.3]{GuiKo} where the authors state  the analogous result  for the sequence space $\ell_1$.

\begin{lemma}
	\label{re-fg-big}
	Let  $ ( \Omega , \Sigma , \mu)  $  be a  measure space.  Assume that $ 0 < \varepsilon < 1$,  $f \in
	B_{L_1(\mu)} $ and  $ g \in B_{L_\infty (\mu)} $ are such that
	$$
	1 - {\varepsilon}^2 < \rea  \int_\Omega fg\ d
	\mu.
	$$
	Then the  set $C$  given by
	$$
	C = \{ t \in \Omega :  \rea  f(t) g(t) > ( 1 - \varepsilon ) \vert f(t) \vert \},
	$$
	satisfies that
	$$
	\rea  \int_C  fg \ d \mu >  1 - \varepsilon.
	$$
\end{lemma}
\begin{proof}
	It is clear that the set $C$ is measurable.  By assumption we have
	\begin{align*}
	1 - {\varepsilon}^2 < \rea  \int_\Omega fg \ d\mu
	\le  \rea  \int_C fg \ d\mu   + ( 1 - \varepsilon) \int_{\Omega \setminus C} |f|   \ d\mu\\
	\le  \varepsilon \rea  \int_C fg \ d\mu   + ( 1 - \varepsilon) \biggl( \int_C \vert f \vert  \ d\mu +\int_{\Omega \setminus C} |f|   \ d\mu \biggr)
	&\le  \varepsilon \rea  \int_C fg \ d\mu  + 1 - \varepsilon.
	\end{align*}
	Hence,
	$$
	\rea  \int_C fg  \ d\mu    > 1 - \varepsilon.
	$$
\end{proof}

\begin{lemma}
	\label{lem-z-mod-z}
	Let  $ z $ be a complex number,  $  0 <\varepsilon < 1 $  and  assume that
	$$
	\rea  z > (1 - \varepsilon) \vert  z\vert .
	$$
	Then
	$$
	\vert  z - \vert  z\vert  \vert   < \sqrt{2\varepsilon} \vert  z\vert .
	$$
\end{lemma}
\begin{proof}
	We write  $z=x+iy$,  where $x,y \in \R$.  Since $ x^2 + y^2 =  \vert z\vert ^2 $ and $ x= \rea z
	> ( 1 - \varepsilon) \vert z\vert $,  we have
	$
	y^2  \leq \vert z\vert ^2 -  (1 - \varepsilon )^2 \vert z\vert ^2  =  ( 2\varepsilon - {\varepsilon}^2 )
	\vert z\vert ^2 $.
	It follows that
	$$
	\bigl\vert z - \vert z\vert   \bigr\vert ^2 =( \vert z\vert -x )^2 + y^2  <   ( \varepsilon \vert z\vert  )^2 + ( 2\varepsilon - {\varepsilon}^2) \vert z\vert ^2 =  2\varepsilon \vert z\vert ^2.
	$$
\end{proof}

We recall the following notion (see for instance \cite[Definition III.3]{DieUh}).

\begin{definition}\label{d-R}
	Let $( \Omega , \Sigma , \mu )$ be a finite measure space and $Y$ a Banach space. An operator $T \in
	\mathcal{L}(L_1(\mu),Y)$ is called {\it Riesz representable} (or simply {\it representable}) if there is $h
	\in L_\infty (\mu, Y)$ such that $T(f)= \int _{\Omega } hf \ d \mu$ for all $f \in L_1 (\mu)$. We say that the function $h$ is   a representation of $T$.
\end{definition}

We will use the following identification.

\begin{proposition} {\rm(\cite[Lemma III.4, p. 62]{DieUh})}
	\label{iden-ope}
	Let  $( \Omega , \Sigma , \mu )$ be a finite measure space and $Y$ be a Banach space. There is a
	linear  isometry  $\Phi$  from  the space  $\mathcal{R}$ of representable operators in
	$\mathcal{L}(L_1(\mu),Y)$ into  $L_\infty ( \mu ,  Y) $ such that if $T \in \mathcal{R}$ and $\Phi(T)=h$, then
	it is satisfied that
	$$
	T(f)= \int _{\Omega } hf \ d \mu, \sem \text{for all } \ \  f \in L_1 (\mu).
	$$
\end{proposition}

It is  known that $\mathcal{WC} (L_1 (\mu))$  is a subset of the representable operators into
$L_1(\mu)$ whenever $\mu$ is any finite measure (see for instance  \cite[Theorem III.12, p.
75]{DieUh}). We will write $\mathcal{R} (L_1(\mu))$ for the space of representable operators into
$L_1(\mu)$. Given $T \in \mathcal{L}(L_1 (\mu))$ and a measurable subset $A$ of $\Omega $, we
will denote by $T_{\vert A}$ the operator on $L_1 (\mu)$ given by $T _{\vert A} (f)= T( f
\chi_A)$ for all $f \in L_1 (\mu)$.

In \cite[Theorem 2.3]{ABGKM} it was proved that a subspace of $\mathcal{L} (L_1 (\mu), Y)$ that
contains the subspace of finite-rank operators and is contained in the space of representable
operators and that satisfies also an additional assumption  has the Bishop-Phelps-Bollob{\'a}s property
for operators whenever $Y$ has the so called AHSp, a property satisfied by $L_1 (\mu)$. Now we
will prove a parallel result for numerical radius for subspaces of  $\mathcal{L} (L_1(\mu))$. Of course,  such
proof is more involved since we have to approximate one pair of elements $(x,x^*)$ in the
product of $S_{L_1 (\mu)} \times  S_{(L_1 (\mu))^*}$ instead of one element in the unit sphere of
$L_1 (\mu)$.

In the  proof of the next result we will write $g(f)$ instead of $\int _{\Omega} g(t) f(t) \ d\mu$ for each
element $f \in L_1 (\mu)$ and $g \in L_\infty ( \mu)$.

\begin{theorem}
	\label{th-BPBP-num-radius-L1}
	Let $(\Omega ,  \mathcal{A} , \mu)$ be a finite measure space and  let $ \mathcal{M} $ be
	a subspace of $ \mathcal{L}( L_1 (\mu) )$ such that $ \mathcal{F}( L_1 (\mu) ) \subseteq
	\mathcal{M} \subseteq \mathcal{R}( L_1 (\mu) ) $. Assume also that  for each measurable subset
	$A$ of $ \Omega  $ and each $ T \in \mathcal{M} $ it is satisfied  $ T_{\vert A} \in
	\mathcal{M}$. Then $ \mathcal{M} $ has the  BPBp-$\nu$, and the function $ \eta $ satisfying   Definition \ref{d-BPBP-num-radius} is independent
	from the measure space and also from $ \mathcal{M} $.
\end{theorem}
\begin{proof}
	Let us fix  $  0 <\varepsilon < 1 $. We take $\eta  \bigl(= \eta (\varepsilon )\bigr)= \frac{\varepsilon^8}{
		2^{33}}$.  Assume that $T_0\in S_{\mathcal{M}}$, $f_0 \in S_{ L_1(\mu) }$ and $ {g_0} \in S_{L_\infty (\mu)}$
	satisfy $g_0(f_0) =1$ and $\vert g_0(T_0(f_0)) \vert > 1-\eta$.
	Let $\lambda _0 $ be a scalar with $\vert \lambda _0 \vert=1$ and such that $ \vert g_0(T_0(f_0)) \vert =\rea
	\lambda _0 g_0(T_0(f_0))$. By changing $T_0$  by $\lambda _0 T_0$  we may assume that $ \rea  g_0(T_0(f_0)) =
	\vert g_0(T_0(f_0)) \vert$. In view of  Proposition \ref{iden-ope} there is a function  $h_0 \in  S_{
		L_\infty (\mu, L_1 (\mu))}$ associated to the operator $T_0$.
	Since the proof is long we divided it into  five steps. \\
	\textbf{Step 1.}  In this step  we will approximate
	the pair of  functions $(f_0, g_0)$ by   a new pair $(f_1, g_1)$ such that $f_1$ and $g_1$ take a countable set of values and also there are subsets where $f_1$, $g_1$ are  constant  and   $h_0$ has small oscillation on these subsets.
	
	More concretely, we will show that there are functions $f_1 \in
	S_{L_1(\mu)}$ and $g_1 \in S_{L_\infty (\mu)}$  and a countable family $\{ D_k: k \in J\} \subset \Omega $ of
	pairwise disjoint measurable sets such that $\mu(D_k)> 0 $ for all $k\in J$,   $ \mu(\Omega\setminus \bigcup_{k \in J  } D_k) = 0$  and  such that  the following conditions are satisfied
	\begin{equation}
	\label{F-aprox}
	\Vert f_1 -f_0\Vert _1  < \frac{\varepsilon}{4},  \ \   \Vert g_1 -g_0\Vert _\infty  < \frac{\varepsilon}{4},
	\end{equation}
	\begin{equation}
	\label{f1-f0-g1-g0}
	\rea g_1(f_1)
	> 1 - \eta,  \ \  \rea g_1(T_0(f_1)) > 1 -\eta,
	\end{equation}
	\begin{equation}
	\label{f1-g1-constant}
	\text{for each} \  k \in J, \  f_1 ~\text{and}~ g_1 ~\text{are constant on} \ D_k
	\end{equation}
	\begin{equation}
	\label{osc-h0-Dk}
	\sup \{ \Vert  h_0 (s) - h_0 (t) \Vert  _1 : s, t \in D_k \}
	\le  \eta , \ \
	\forall k \in J,
	\end{equation}
	and
	\begin{equation}
	\label{norm-h0}
	1=\Vert h_0 \Vert _\infty = \sup \{ \Vert  h_0 (t)\Vert  _1 : t \in  \cup_{k \in J} D_k  \}.
	\end{equation}

	Since the set of simple functions is  dense in both $L_1(\mu)$ and $L_{\infty}(\mu)$, there are   simple
	functions $ f_1 \in S_{L_1(\mu)}$ and $g_1 \in S_{L_\infty(\mu)}$ satisfying \eqref{F-aprox} and \eqref{f1-f0-g1-g0}.
	
	On the other hand, by \cite[Theorem II.2, p. 42]{DieUh} there is a measurable subset $E_1$  of
	$\Omega$ such that $\mu(E_1)= 0 $ and $ h_0 ({\Omega\setminus E_1}) $ is a separable subset of
	$L_1(\mu)$. Suppose that  the set $\{y_i : i\in \mathbb{N}\}$ is dense in  $ h_0 ({\Omega\setminus E_1})$. Since $f_1$ and $g_1$ are simple functions, we can assume that $\text{Im}(f_1)=\{a_r : r=1, \ldots, n\}$ and $\text{Im}(g_1)=\{b_l : l=1, \ldots, m\}$. Now, for $i\in \mathbb{N}$, $r\in\{ 1, \ldots,n\}=N$ and $l\in \{1,\ldots, m\}=M$ we  consider  the following subsets of $\Omega $
	$$
	A_{(1, r , l)}=h_0^{-1}(B_{\frac{\eta}{2}}(y_1))\cap({\Omega\setminus E_1})\cap f_1^{-1}(a_r)\cap g_1^{-1}(b_l)
	$$
	and
	$$
	A_{(i, r, l)}= (h_0^{-1}(B_{\frac{\eta}{2}}(y_i))\setminus \cup_{e=1}^{i-1}h_0^{-1}(B_{\frac{\eta}{2}}(y_e)))\cap({\Omega\setminus E_1})\cap f_1^{-1}(a_r)\cap g_1^{-1}(b_l), \quad \forall i\geq 2.
	$$
	It is clear that the elements of the family $\{A_{(i, r, l)} : (i, r, l)\in \mathbb{N}\times N\times M\}$ are measurable  subsets of $\Omega$   and pairwise disjoint. Now, let $W=\{(i, r, l)\in  \mathbb{N} \times N\times M : \mu(A_{(i, r, l)})=0\}$ and $E_2=\bigcup_{(i, r, l)\in W} A_{(i, r, l)}$.  By  the definition of $W$ it is trivially satisfied that $E_2$ is measurable and   $\mu(E_2)=0$.
	On the other hand  there exists a measurable subset $E_3$  of $\Omega\setminus (E_1\cup E_2)$ such that $\mu(E_3)=0$ and $
	\Vert h \Vert _\infty = \sup \{ \Vert h(t) \Vert_1 : t \in \Omega \backslash E_3\}.$  Assume that $\{D_k : k\in J\}$ is the family of pairwise disjoint measurable subsets obtained by indexing the set $\{A_{(i, r, l)}\setminus E_3 : (i, r, l)\in (\mathbb{N}\times N\times M)\setminus W\}$.
	Then, we have that $\mu(D_k)> 0$ for all $k\in J$,   $ \mu(\Omega\setminus \bigcup_{k \in J  } D_k) = 0$ and also the family $\{D_k : k\in J\}$  satisfies  the conditions \eqref{f1-g1-constant}, \eqref{osc-h0-Dk} and \eqref{norm-h0}. Therefore, by \eqref{f1-g1-constant} there are sets of scalars $\{ \alpha _k : k \in J \}$ and $\{ \gamma  _k : k \in J \}$   such that
	\begin{equation}
	\label{f1-g1}
	f_1=\sum_{k \in J} \alpha_k \frac{\chi_{D_k}}{\mu(D_k)} , \ \  \sum_{k \in J} \vert \alpha _k \vert =1,\sem
	g_1=\sum_{k \in J} \gamma _k \chi_{D_k} , \ \   \vert \gamma _k \vert \le 1, \sem \forall k \in J.
	\end{equation}
	
	\textbf{Step 2.}  In this step we  will define another   simple function $f_2\in S_{L_1(\mu)}$ which is an approximation of $f_1$,  and can be expressed as a finite sum  instead of the countable sum appearing in the expression of $f_1$ given in \eqref{f1-g1}.\\
	
	By  \eqref{f1-g1}  and \eqref{f1-f0-g1-g0}	there is a finite subset $F$ of $J$  such that
	\begin{equation}
	\label{sum-F-alpha}
	\sum _{k \in F} | \alpha_k | > 1- \eta > 0, \seg   \rea  g_1\biggl(  \sum_{k \in F} \alpha_k
	\frac{\chi_{D_k}}{\mu(D_k)} \biggr)  > 1 - \eta,
	\end{equation}
	and also
	\begin{equation}
	\label{g1-T0-sum-F}
	\rea  g_1\biggl(T_0 \biggl(  \sum_{k \in F} \alpha_k \frac{\chi_{D_k}}{\mu(D_k)} \biggr) \biggr)
	> 1 -  \eta .
	\end{equation}
	
	For each $k \in F$ we put  $\beta _k = \frac{\alpha_k}{  \sum_{k \in F} \vert \alpha_k \vert } $ and define
	$f_2 = \sum_{k \in F} \beta_k \frac{\chi_{D_k}}{\mu(D_k)} $. In view of \eqref{sum-F-alpha} and
	\eqref{g1-T0-sum-F} we have that
	\begin{equation}
	\label{g1-f2}
	\rea  g_1(f_2)= \rea  g_1\biggl(  \sum_{k \in F} \beta_k \frac{\chi_{D_k}}{\mu(D_k)} \biggr)
	> 1 - \eta
	\end{equation}
	and
	\begin{equation}
	\label{g1-T0-sum-F-norm}
	\rea  g_1(T_0(f_2)) = \rea  g_1\biggl(T_0 \biggl(  \sum_{k \in F} \beta_k \frac{\chi_{D_k}}{\mu(D_k)}
	\biggr) \biggr)  > 1 -  \eta .
	\end{equation}

	Clearly  $f_2 \in S_{ L_1 (\mu) }$ and by \eqref{f1-g1},  \eqref{sum-F-alpha} we have that
	\begin{align}
	\label{f2-f1}
	\nonumber \Vert f_2 - f_1  \Vert _1 & =  \biggl \Vert  \sum_{ k \in  F} \beta_k \frac{\chi_{D_k}}{\mu(D_k)}
	-
	\sum_{k \in J} \alpha_k \frac{\chi_{D_k}} {\mu (D_k)} \biggr \Vert _1\\
	\nonumber
	& = \biggl \Vert  \sum_{ k \in  F} \beta_k \frac{\chi_{D_k}}{\mu(D_k)}  -
	\sum_{k \in F} \alpha_k \frac{\chi_{D_k}} {\mu (D_k)}   -   \sum_{k \in  J \backslash F} \alpha_k \frac{\chi_{D_k}} {\mu (D_k)}       \biggr \Vert _1\\
	&\le  \sum_{k \in F} |\beta_k - \alpha_k| + \sum_{ k \in  J \setminus F}
	|\alpha_k|
	%
	= 1 -  \sum_{k \in F} | \alpha_k| + \sum_{ k \in J \setminus F}
	|\alpha_k| \\
	\nonumber
	&=  2  \Bigl( 1 -  \sum_{k \in F} | \alpha_k|  \Bigr)
	%
	<  2\eta < \frac{\varepsilon}{4}.
	\end{align}
	
	\textbf{Step 3.} Now, we approximate   the function $h_0$ by a new one  $h_2$  such that for each $k \in F$  the new function is constant on each $D_k$.
	So we  also   approximate the operator  $T_0$ by a new one.\\
	For this aim   we choose an element  $t_k$ in $ D_k$, for any $k \in F $,  put $\psi_k  = h_0(t_k)\in L_1 (\mu) $ and define  $h_1 \in L_\infty(\mu, L_1 (\mu))$ by
	$$
	h_1 = h_0 \chi_ {\Omega \setminus (\bigcup_ {k \in F} D_k) } +
	\sum _{ k \in F} \psi_k  \chi_{D_k}.
	$$
	By \eqref{norm-h0} we have that $\Vert h_1 \Vert _\infty \le 1$.  If $T_1\in \mathcal{L} (L_1	(\mu))$ is the  operator associated to  $h_1$, then $T_1$ is the sum of ${T_0}_{\vert \Omega \setminus (\bigcup_ {k \in F} D_k) }$ and a finite-rank operator, so $T_1 \in B_{\mathcal{M}}$. By using   \eqref{osc-h0-Dk}, we
	clearly have
	\begin{align}
	\label{T1-T0}
	\Vert T_1 - T_ 0 \Vert   &=  \Vert  h_1 - h_0 \Vert _\infty \le \sup \{ \Vert  \psi_k  - h_0(t) \Vert _1  :  t \in D_k, k \in F \}  \\
	\nonumber
	&= \sup  \{ \Vert  h_0(t_k) - h_0(t) \Vert _1  :  t \in D_k, k \in F \} \le    \eta.
	\end{align}
	Since $\Vert T_0 \Vert=1$ we get that $ 0<1-\eta  \le  \Vert  T_1 \Vert \le 1$.   Now we define $ T_2=	\frac{T_1}{\Vert {T_1}\Vert} $ and so
	we have that
	$$
	\Vert T_2 - T_1 \Vert = 1 - \Vert T_1 \Vert \le  \eta .
	$$
	In view of the previous inequality and \eqref{T1-T0} we obtain that
	\begin{equation}
	\label{T2-T0}
	\Vert T_2 - T_0 \Vert \le \Vert T_2 - T_1 \Vert + \Vert T_1 - T_0 \Vert \le  2 \eta  <
	\frac{\varepsilon}{4}.
	\end{equation}
	From \eqref{g1-T0-sum-F-norm} and \eqref{T2-T0}   we get that
	\begin{align}
	\label{g1-T2-f2}
	\rea g_1(T_2 (f_2)) \ge  \rea g_1 ( T_0 (f_2)) -  \Vert T_2 - T_ 0 \Vert  > 1-3\eta  .
	\end{align}
	
	On the other hand, it is clear that
	\begin{align*}
	T_1 (f_2)  = \int _ \Omega  h_1 f_2 \ d \mu   = \int _{\Omega\setminus \bigcup_{k \in F} D_k}  h_1 f_2 \ d \mu   +
	\sum _ {k \in F}  \int _ {D_k} h_1 f_2 \ d\mu = \sum _ {k \in F}   \beta _k \psi_k.
	\end{align*}
	For simplicity,  for each $ k\in F $,  put $ \phi_k =\dfrac{\psi_k }{\Vert T_1 \Vert}$.  So we have 	that
	$$
	T_2(f_2) = \sum_{k\in F} \beta_k \phi_k .
	$$
	It is clear that $\phi_k \in B_{L_1(\mu)}$ for every $k \in F$.
	From \eqref{g1-f2} and \eqref{g1-T2-f2} we obtain that
	\begin{align*}
	\rea g_1  \biggl(  \sum _ {k \in F}\frac{ \beta_k}{2}  \biggl( \frac{\chi_{D_k}}{\mu (D_k)} + \phi_k \biggr)
	\biggr)  =
	\rea g_1 \biggl( \frac{f_2 + T_2(f_2)}{2} \biggr)
	>  1-  2 \eta  .
	\end{align*}
	
	\textbf{Step 4.} In this step   we will  obtain  approximations  $f_3$, $T_3$ of $f_2$ and $T_2$, respectively. We will check  in the final step that $T_3$ attains its norm at $f_3$, a necessary condition for our purpose. In fact $f_3$ and $T_3$ are the final approximations to $f_0$ and $T_0$.\\
	Define the set $G$ as follows
	$$
	G = \biggl\{ k \in  F : \rea     g_1 \biggl( \frac{\beta _k}{2} \biggl( \frac{\chi_{D_k}}{\mu
		(D_k)} + \phi_k \biggr) \biggr)  > \bigl( 1 -\sqrt{ 2 \eta } \bigr) |\beta_k| \biggr\}.
	$$
	In view of Lemma \ref{lema-compl-z-beta} we have that
	\begin{equation}
	\label{sum-G}
	\sum_{k \in G } | \beta_k | >  1 - \sqrt{2 \eta  } = 1 - \frac{\varepsilon^4}{2^{16}}.
	\end{equation}
	It is  immediate  that
	$$
	\rea  \beta_k g_1 \biggl( \frac{\chi_{D_k}}{\mu ({D_k})} \biggr) > \Bigl( 1 - 2 \sqrt{2 \eta} \Bigr) |
	\beta_k | = \Bigl( 1 - \frac{ \varepsilon^4}{2^{15}} \Bigr) | \beta_k |, \sem \forall k \in G.
	$$
	So, for each  $ k \in G $  we have
	\begin{align*}
	\rea  \beta_k \gamma_k= \rea  \beta_k g_1 \biggl( \frac{\chi_{D_k}}{\mu(D_k)}\biggr) > \Bigl( 1 - \frac{ \varepsilon^4}{2^{15}} \Bigr) | \beta_k |
	\ge  \Bigl( 1 - \frac{ \varepsilon^4}{2^{15}} \Bigr) | \beta_k  \gamma _k|.
	\end{align*}
	Hence, we obtain that $ \beta_k\neq 0$  for $k \in G$ and  also that
	\begin{equation}
	\label{gamma-k-big}
	\vert \gamma _k \vert  > 1 - \frac{\varepsilon^4}{ 2^{15}}>0, \seg \forall k \in G.
	\end{equation}
	By using also  Lemma \ref{lem-z-mod-z} we get
	$$
	| \beta_k \gamma_k - | \beta_k \gamma_k | | < \frac{\varepsilon^2}{2^7} |\beta_k \gamma_k |.
	$$
	Hence,
	\begin{equation}
	\label{beta-k-gamma-k}
	\biggl \vert  \beta_k -  \frac { \vert \beta_k \gamma_k \vert}{\gamma_k} \biggr \vert  <
	\frac{\varepsilon^2}{2^7} \vert \beta_k  \vert   ~~\text{and} ~~ \biggl \vert  \gamma_k - \frac { \vert
		\beta_k \gamma_k \vert}{\beta_k} \biggr \vert  <  \frac{\varepsilon^2}{2^7}  \vert \gamma_k \vert, \sem
	\forall k \in G,
	\end{equation}
	so
	\begin{equation}
	\label{gamma-k-norm}
	\biggl \vert \frac{\gamma_k}{\vert \gamma_k \vert}  -  \frac{\vert \beta_k \vert }{\beta_k} \biggr \vert
	<  \frac{\varepsilon^2}{2^7},  \seg \forall k \in G .
	\end{equation}

	The element  $ f_3 $  given by
	$$
	f_3 =  \frac{1}{ \sum_{k \in G}  \vert \beta _k \vert} \sum_{ k \in G} \frac{ \vert \beta _k \gamma_k \vert
	}{ \gamma _k} \frac{\chi_{D_k}}{\mu(D_k)}
	$$
	belongs to the unit sphere of $L_1 (\mu)$.  Now, by using  \eqref{sum-G} and  \eqref{beta-k-gamma-k} we get that
	
	\begin{align}
	\label{f3-f2}
	\nonumber
	\Vert  f_3 - f_2    \Vert _1 &= \biggl \Vert  \frac{1}{ \sum_{k \in G}  \vert \beta _k \vert} \sum_{ k \in G} \frac{ \vert
		\beta _k \gamma_k \vert  }{ \gamma _k} \frac{\chi_{D_k}}{\mu(D_k)}-
	\sum_{k \in F}   \beta _k  \frac{\chi_{D_k}}{\mu(D_k)} \biggr \Vert_1 \\
	\nonumber
	&=  \biggl \Vert  \frac{1}{ \sum_{k \in G}  \vert \beta _k \vert} \sum_{ k \in G} \frac{ \vert \beta _k
		\gamma_k \vert  }{ \gamma _k} \frac{\chi_{D_k}}{\mu(D_k)}-
	\sum_{k \in G}   \beta _k  \frac{\chi_{D_k}}{\mu(D_k)} -
	\sum_{k \in F \backslash  G}   \beta _k  \frac{\chi_{D_k}}{\mu(D_k)} \biggr \Vert_1  \\
	&\le
	\sum_{ k \in G}  \biggl \vert   \frac{1}{\sum _{ k \in G } \vert \beta _k \vert}
	\frac{ \vert \beta _k \gamma_k \vert }{  \gamma _k} - \beta _k \biggl \vert +
	\sum_{k \in F \backslash  G}   \vert \beta _k \vert
	\\
	\nonumber
	&\le
	\sum_{ k \in G}  \biggl \vert   \frac{1}{\sum _{ k \in G } \vert \beta _k \vert}
	\frac{ \vert \beta _k \gamma_k \vert }{  \gamma _k} - \frac{ \vert \beta _k \gamma_k \vert }{  \gamma _k}
	\biggl \vert +
	\sum_{ k \in G}  \biggl \vert   \frac{ \vert \beta _k \gamma_k \vert }{  \gamma _k} -  \beta _k \biggr \vert
	+
	\sum_{k \in F \backslash  G}   \vert \beta _k \vert   \\
	\nonumber
	&\le
	1- \sum _{k \in G} \vert \beta _k \vert +   \sum _{k \in G}  \frac{\varepsilon^2}{2^7} \vert \beta _k \vert +
	\sum_{k \in F \backslash  G}   \vert \beta _k \vert\\
	\nonumber
	&\le
	2 \Bigl( 1- \sum _{k \in G} \vert \beta _k \vert \Bigr)  +   \frac{\varepsilon^2}{2^7}  \le   \frac{\varepsilon}{8}.
	\end{align}
	In view of \eqref{F-aprox},  \eqref{f2-f1} and \eqref{f3-f2},  we obtain that
	\begin{align}
	\label{f3-f0}
	\Vert  f_3 - f_0 \Vert _1 &\le   \Vert  f_3 - f_2 \Vert _1+ \Vert  f_2 - f_1 \Vert _1 + \Vert  f_1 - f_0 \Vert _1 <     \frac{\varepsilon}{8} +  \frac{\varepsilon}{4} +  \frac{\varepsilon}{4} < \varepsilon.
	\end{align}

	Now notice obviously  that
	$$
	\rea  \beta_k g_1(\phi_k ) > \bigl( 1 - 2 \sqrt{2\eta } \bigr)  | \beta_k |
	> \Bigl( 1 - \frac{ \varepsilon^4}{ 2^{14} } \Bigr)  | \beta_k |
	, \seg \forall k \in G .
	$$

	For each $ k \in G $, define $P_k $ as follows
	$$
	P_k  = \Bigl\{ t \in \Omega :   \rea    \beta_k g_1(t) \phi_k (t)  >
	\Bigl( 1 - \frac{ \varepsilon^2}{ 2^{7} } \Bigr)
	\vert \beta_k {\phi_k  (t)} \vert  \Bigr\}.
	$$

	Clearly $P_k$ is a measurable set.  According to Lemma  \ref{re-fg-big},  for each $ k \in G $
	we have
	$$
	\rea  \int_{P_k } \beta_k g_1 \phi_k \ d \mu >
	\Bigl( 1 - \frac{ \varepsilon^2}{ 2^{7} } \Bigr)  \vert \beta _k \vert,
	$$
	so
	\begin{equation}
	\label{int-Pk}
	\int_{P_k} |\phi_k | d \mu   > 1-  \frac{\varepsilon^2}{2^7} > 0.
	\end{equation}
	
	Let us fix $k \in G$ and $t \in P_k$. Notice that $\beta _k g_1(t) \ne 0$. By Lemma
	\ref{lem-z-mod-z} it follows
	$$
	\bigl \vert  \beta_k g_1(t)  \phi_k (t) -  \vert  \beta_k g_1(t) \phi_k (t) \vert \; \bigr \vert < \frac{
		\varepsilon}{2^3}    \vert  \beta_k g_1(t) \phi_k (t) \vert,
	$$
	so
	\begin{equation}
	\label{phi-k-beta-k-g1}
	\biggl \vert   \phi_k  (t) -  \frac{  \vert \beta_k  g_1(t) \phi_k (t) \vert  } { \beta_k g_1(t)
	}    \biggr \vert   <  \frac{ \varepsilon}{2^3} \vert \phi_k (t) \vert , \seg \forall k \in G, t
	\in P_k.
	\end{equation}
	
	For each $ k \in G $  we   can define the element  $ \varphi_k $ in $L_1 (\mu)$ by
	$$
	\varphi_k  =  \frac{\gamma_k}{ \vert \gamma_k \vert }  \frac{\vert \phi_k  \vert}{ \int_{P_k } \vert \phi_k
		\vert \ d \mu}   \frac{ \vert g_1 \vert}{g_1} \chi_{P_k} .
	$$
	It is immediate that $ \varphi_k  \in S_{ L_1 (\mu)}$.    From   \eqref{int-Pk} and \eqref{phi-k-beta-k-g1}, for  each $ k \in G $ we have
	\begin{align}
	\label{varphi-k-phi-k}
	\nonumber
	\Vert      \varphi_k  - \phi_k  \Vert _1&\le \Vert \varphi_k  - \phi_k
	\chi_{P_k }\Vert _1 +  \Vert  \phi_k  \chi_{\Omega \setminus P_k  }  \Vert _1  \\
	\nonumber
	& < \biggl \Vert  \varphi_k  -
	\phi_k  \chi_{P_k} \biggr \Vert _1  +   \frac{ \varepsilon^2}{2^7} \\
	\nonumber &\le \biggl \Vert  \varphi_k  -
	\frac{\gamma_k}{ \vert \gamma_k  \vert} \vert \phi_k \vert   \frac{\vert g_1 \vert}{g_1} \chi
	_{P_k} \biggr \Vert _1  +
	\biggl \Vert \frac{\gamma_k}{ \vert \gamma_k  \vert}  \vert  \phi_k \vert    \frac{\vert g_1
		\vert }{g_1} \chi _{P_k} - \frac{ \vert \beta _k \vert}{ \beta _k}  \vert  \phi_k \vert
	\frac{\vert g_1 \vert}{g_1} \chi _{ P_k} \biggr \Vert _1    \\
	\nonumber &+ \biggl \Vert  \frac{ \vert \beta _k \vert}{ \beta _k} \vert  \phi_k  \vert
	\frac{\vert g_1 \vert}{g_1} \chi _{ P_k} -  \phi_k   \chi_{P_k} \biggr \Vert _1 +  \frac{ \varepsilon^2}{2^7} \\
	&\le \biggl \Vert  \varphi_k  -
	\frac{\gamma_k}{ \vert \gamma_k  \vert} \vert  \phi_k \vert    \frac{\vert g_1 \vert}{g_1} \chi
	_{P_k} \biggr \Vert _1 + \biggl \vert  \frac{\gamma _k}{ \vert \gamma _k \vert }  - \frac{ \vert
		\beta _k \vert } {\beta_k} \biggr \vert +  \frac{ \varepsilon}{2^3} +   \frac{
		\varepsilon^2}{2^7}
	\\
	\nonumber &\le \biggl \Vert  \varphi_k  -
	\frac{\gamma_k}{ \vert \gamma_k  \vert} \vert  \phi_k \vert  \frac{\vert g_1 \vert}{g_1} \chi
	_{P_k} \biggr \Vert _1 + \frac{ \varepsilon}{4}
	\ \  \text{(by \eqref{gamma-k-norm})}\\
	%
	\nonumber &=  \biggl \Vert   \frac{\gamma _k}{ \vert   \gamma _k \vert }   \frac{\vert \phi_k
		\vert}{ \int _{P_k} \vert \phi_k  \vert \ d \mu }  \frac{ \vert g_1\vert}{g_1} \chi_{P_k}-
	\frac{\gamma_k}{ \vert \gamma_k \vert} \vert  \phi_k \vert   \frac{\vert g_1 \vert}{g_1} \chi
	_{P_k} \biggr \Vert _1 + \frac{ \varepsilon}{4}\\
	%
	\nonumber &=   1- \int _{P_k} \vert \phi_k  \vert \ d \mu   + \frac{ \varepsilon}{4} <   \frac{ \varepsilon^2}{2^7} +   \frac{  \varepsilon}{4}  < \frac{\varepsilon}{2}.
	\end{align}
	
	Let the function $h_3$ be defined as follows
	$$
	h_3 = \frac{ h_1 }{ \Vert h_1 \Vert _\infty}  \chi_{ \Omega \setminus \bigcup_{ k \in G} D_k}  + \sum_{k \in
		G } \varphi_k  \chi_{D_k}.
	$$
	It is easy to see that $h_3$ belongs to the unit sphere of $L_\infty (\mu, L_1 (\mu))$. Let $T_3 \in  S_{\mathcal{L} (L_1
		(\mu))}$ be the operator associated to the function $h_3$ in view of Proposition \ref{iden-ope}.  Since $G$
	is a finite set, $\mathcal{F} (L_1 (\mu))  \subset \mathcal{M}$ and $T_1 \in \mathcal{M}$, by using the
	assumptions on $\mathcal{M}$ we know that $T_3 \in S_{\mathcal{M}}$.
	
	We also have  that
	\begin{align*}
	\Vert T_3 - T_2 \Vert &= \biggl\Vert  h_3 - \frac{h_1} {\Vert h_1 \Vert} _\infty \biggr \Vert _\infty \\
	& =    \biggl\Vert  h_3 \chi _{ \Omega \backslash (\cup_{ k \in G} D_k) } +
	\sum_{ k \in G} h_3 \chi_{D_k}
	- \frac{h_1} {\Vert h_1 \Vert} _\infty \chi _{ \Omega \backslash (\cup_{ k \in G} D_k) }
	- \sum _{k \in G} \frac{h_1} {\Vert h_1 \Vert}
	_\infty\chi_{D_k}  \biggr \Vert _\infty  \\
	& =    \biggl\Vert   \frac{h_1} {\Vert h_1 \Vert} _\infty   \chi _{ \Omega \backslash (\cup_{ k \in G} D_k) }
	+
	\sum_{ k \in G}  \varphi_k  \chi_{D_k}
	- \frac{h_1} {\Vert h_1 \Vert} _\infty \chi _{ \Omega \backslash (\cup_{ k \in G} D_k )}
	- \sum _{k \in G} \phi_k  \chi_{D_k}  \biggr \Vert _\infty  \\
	&=  \Bigl\Vert    \sum_{ k \in G}  \bigl( \varphi_k-  \phi_k \bigr)  \ \chi_{D_k}  \Bigr \Vert_\infty = \sup_{ k \in G} \bigl \Vert  \varphi_k -  \phi_k  \bigr \Vert _1   \le   \frac{\varepsilon}{2}  \sem \text{(by \eqref{varphi-k-phi-k})}.
	\end{align*}

	By the previous inequality and \eqref{T2-T0} we obtain
	\begin{equation}
	\label{T3-T0}
	\Vert  T_3 - T_0  \Vert  \le     \Vert  T_3 - T_2  \Vert +  \Vert  T_2 - T_0  \Vert  < \varepsilon.
	\end{equation}
	
	\textbf{Step 5.} Finally, we  are going to find an approximation of $g_1$ and complete our proof.\\
	
	We put $A= \bigl\{ t \in \Omega : \vert g_1 (t) \vert \ge 1 - \frac{\varepsilon^2}{2^7} \bigr\}$
	and let the function $g_2 $  be defined by $ g_2 = \dfrac{g_1}{\vert g_1 \vert} \chi_A + g_1 \chi
	_{ \Omega \backslash A}$. Since $g_1 \in S_{ L_\infty (\mu)}$, we have that    $g_2 \in S_{
		L_\infty (\mu)}$. It is also clear that
	\begin{equation}
	\label{g2-g1}
	\Vert g_2- g_1 \Vert _ \infty  \le \frac{\varepsilon^2}{2^7}.
	\end{equation}
	
	By using \eqref{F-aprox} and \eqref{g2-g1}  we also have that
	\begin{equation}
	\label{g2-g0}
	\Vert g_2- g_0 \Vert _\infty \le \Vert g_2- g_1 \Vert _\infty +  \Vert  g_1 - g_0 \Vert _\infty \le
	\frac{\varepsilon^2}{2^7}+ \frac{\varepsilon}{4} < \varepsilon.
	\end{equation}

	By \eqref{gamma-k-big} we know that   $\vert \gamma _k \vert > 1- \frac{\varepsilon^4}{ 2^{15}}$ for each $k
	\in G$. Since $G \subset J$, in view of  \eqref{f1-g1},  the restriction of $g_1$ to $D_k$ coincides with
	$\gamma _k$ and so $D_k\subset A$ for all $k\in G$. Hence,
	$$
	{g_2}_{\vert D_k} = \dfrac{\gamma _k}{\vert \gamma _k \vert } , \seg \forall
	k \in G.
	$$
	Therefore,  we deduce that
	\begin{align}
	\label{g2-f3}
	\nonumber
	g_2 (f_3)&= g_2 \biggl( \frac{1}{ \sum_{k \in G}  \vert \beta _k \vert}
	\sum_{ k \in G} \frac{ \vert \beta _k \gamma_k \vert  }{
		\gamma _k} \frac{\chi_{D_k}}{\mu(D_k)}\biggr)  \\
	& =    \frac{1}{ \sum_{k \in G}  \vert \beta _k \vert}   \sum_{ k \in G} \frac{ \vert \beta _k \gamma_k \vert
	}{
	\gamma _k} \frac{1}{\mu(D_k)}  g_2 \bigl( \chi_{D_k} \bigr)  \\
\nonumber
&=    \frac{1}{ \sum_{k \in G}  \vert \beta _k \vert}   \sum_{ k \in G} \frac{ \vert \beta _k
	\gamma_k \vert }{
	\gamma _k}   \frac{ \gamma _k }{ \vert \gamma _k \vert }   =1.
\end{align}

For each $k \in G$, from the  definition of $P_k$ and $A$, we deduce that $P_k \subset A$, so
\begin{equation}
\label{g2-varphi-k}
g_2(   \varphi_k ) = \int _{P_k} \frac{\gamma _k}{ \vert \gamma _k \vert}  \frac{\vert \phi_k
	\vert }{ \int _{P_k} \vert \phi_k \vert \ d  \mu } \ d \mu = \frac{\gamma _k }{\vert \gamma _k
	\vert }.
\end{equation}
Since
$$
T_3(f_3) = \int_{\Omega} h_3 f_3\ d \mu  =  \frac{1}{ \sum _{ k \in G} \vert \beta _k \vert} \sum_{k \in G}
\frac{ \vert \beta _k \gamma _k \vert }{\gamma_k} \varphi_k,
$$
by using 	\eqref{g2-varphi-k} we have that
\begin{align}
\label{g2-T3-f3}
\nonumber
g_2(T_3(f_3))  & =  \frac{1}{ \sum _{ k \in G} \vert \beta _k \vert} \sum_{k \in G} \frac{ \vert
	\beta _k \gamma _k \vert }{\gamma_k}
g_2 \bigl( \varphi_k \bigr)  \\
& =  \frac{1}{ \sum _{ k \in G} \vert \beta _k \vert} \sum_{k \in G} \frac{ \vert \beta _k \gamma
	_k \vert }{\gamma_k} \frac{\gamma _k}{ \vert \gamma _k \vert} = 1.
\end{align}

We  have shown that there are   elements  $ T_3 \in S_{\mathcal{M}}$, $ f_3 \in S_{L_1 (\mu)}$ and $ g_2 \in
S_{L_\infty (\mu)} $ that in view of \eqref{f3-f0}, \eqref{T3-T0},   \eqref{g2-g0}, \eqref{g2-f3} and
\eqref{g2-T3-f3} satisfy
$$
\Vert  T_3 - T_0  \Vert  < \varepsilon , \seg \Vert  f_3 - f_0  \Vert_1   < \varepsilon ,\seg
\Vert g_2 - g_0 \Vert_\infty  < \varepsilon
$$
and also
$$
g_2 (f_3)  =  g_2 ( T_3(f_3) )  = 1.
$$
So  we showed that $ \mathcal{M} $  has the  BPBp-$\nu$  with the function $ \eta $
given by
\linebreak[4]
$ \eta (\varepsilon) = \dfrac{\varepsilon^8}{2^{33}}$.
\end{proof}

In case that $\mu$ is a  $\sigma$-finite measure, there is a finite measure  $\zeta$  and a  linear isometry $\Phi$ from $L_1
(\mu)$ onto $L_1 (\zeta)$.
From this fact we deduce the following result which generalizes Theorem \ref{th-BPBP-num-radius-L1}  for some well-known classes of operators.

\begin{corollary}
	\label{cor-L1}
	Let $ ( \Omega , \Sigma , \mu   )  $ be a $\sigma$-finite measure space. The  following
	subspaces of $ \mathcal{L}  ( L_1 (\mu))  $ have the  BPBp-$\nu$ and the function $\eta $ satisfying Definition \ref{d-BPBP-num-radius} is
	independent from the  measure space.
	\begin{enumerate}
		\item[1)]  The subspace  of all finite-rank operators  on $L_1 (\mu) $.
		\item[2)]  The subspace  of all compact operators  on $L_1 (\mu) $.
		\item[3)]  The subspace  of all weakly compact operators  on $L_1 (\mu) $.
	\end{enumerate}
	In case that $\mu $ is finite, then the subspace   of all   representable operators  on  $L_1
	(\mu)$ also has the BPBp-$\nu$.
\end{corollary}
\begin{proof}
	Assume first  that $\mu$ is a finite measure.  It is known that   $\mathcal{F}(L_1(\mu)) \subset \mathcal{K}(L_1(\mu)) \linebreak[4] \subset \mathcal{WC}(L_1(\mu))\subset \mathcal{R}(L_1(\mu))$ and  $T_{\vert A}(B_{L_1(\mu)})\subset T(B_{L_1(\mu)})$ for each  $T \in {\mathcal{L}} (L_1 (\mu))$ and every measurable subset
	$A$ of $ \Omega$. Also, it is clear that $T_{\vert A}\in \mathcal{R}(L_1(\mu))$ for any $T\in \mathcal{R}(L_1(\mu))$  and every measurable subset
	$A$ of $ \Omega$. Therefore, the spaces $\mathcal{F}(L_1(\mu))$, $\mathcal{K}(L_1(\mu))$,  $\mathcal{WC}(L_1(\mu)) $ and $ \mathcal{R}(L_1(\mu))$ satisfy the assumptions of Theorem \ref{th-BPBP-num-radius-L1}, and so the above statements hold in case that $\mu$ is finite.
	
	Now, let $\mu$ be a  $\sigma$-finite measure. We will show that the space
	$\mathcal{F}(L_1(\mu)) $ satisfies the BPBp-$\nu$.
	There is a finite measure $\zeta$  and a  surjective linear isometry  $\Phi$ from $L_1(\mu)$ into $L_1 (\zeta)$.   The mapping $\Phi$ induces a surjective linear isometry from
	$\mathcal{F}(L_1(\mu)))$ into $\mathcal{F}(L_1(\zeta)))$ given by  $T \mapsto \Phi \circ T \circ \Phi^{-1}$.   Since $\Phi$  is an isometry, it follows that
	$\nu (T) = \nu ( \Phi \circ T \circ \Phi^{-1})$ for every $T \in \mathcal{F}(L_1(\mu))$.
	On the other hand, it is satisfied that $(f,g) \in \Pi (L_1 (\mu))$ if and only if
	$(\Phi(f), (\Phi^{-1}) ^t (g) ) \in \Pi (L_1 (\zeta))$.  Also $(\Phi^{-1})^t(g)( \Phi \circ T \circ \Phi^{-1}(\Phi(f)))=g(T(f))$ for every $T \in \mathcal{F}(L_
	1(\mu)))$.
	Since   $\mathcal{F}(L_1(\zeta))$ has the BPBp-$\nu$ we deduce the same property for $\mathcal{F}(L_1(\mu))$.
	
	The  proofs of the statements  2) and 3)    are  analogous.
\end{proof}

\vskip 5mm

\vspace{3mm}

{\bf Acknowledgements.}
The authors would like to thank
the  reviewer for  valuable comments. The research work of the third author was done during her visit to  University of Granada. She thanks the Department of Mathematical Analysis and  the International Welcome Center of University of Granada,  and specially wishes to thank Prof. Mar\'{\i}a D. Acosta, for kind hospitality.

\end{document}